\documentclass{amsart}
\usepackage{amsmath}
  \usepackage{paralist}

\newcommand{\cW}{\mathcal{W}}
\newcommand{\cT}{\mathcal{T}}

\newcommand{\mR}{\mathbb{R}}

\newcommand{\tr}{\operatorname{tr}}

\newtheorem{theorem}{Theorem}[section]
\newtheorem{corollary}{Corollary}

\newtheorem{lemma}[theorem]{Lemma}

\theoremstyle{definition}
\newtheorem{definition}[theorem]{Definition}
\newtheorem{remark}{Remark}

\title[Nonlinear thermoelasticity]
      {Maximal regularity and global existence of solutions
  to a quasilinear  thermoelastic plate  system }

\author[Irena Lasiecka and Mathias Wilke]{}

\subjclass[2010]{Primary: 74F05; Secondary: 35B30, 35B40, 74H40.}
 \keywords{Quasilinear thermoelastic plates,
existence and uniqueness of strong solutions, maximal regularity, exponential decay.}

 \email{il2v@virginia.edu}
 \email{mathias.wilke@mathematik.uni-halle.de}

\thanks{The research of I. Lasiecka has been partially supported by DMS-NSF
Grant Nr 0606882. M. Wilke expresses his thanks for hospitality to the Department for Mathematics at the University of Virginia.}

\begin{document}
\maketitle

\centerline{\scshape Irena Lasiecka }
\medskip
{\footnotesize
 \centerline{Department of Mathematics}
   \centerline{University of Virginia}
   \centerline{Charlottesville, VA 22903, USA}
} 

\medskip

\centerline{\scshape Mathias Wilke}
\medskip
{\footnotesize
 \centerline{ Institut f\"{u}r Mathematik}
   \centerline{Martin-Luther-Universit\"{a}t Halle-Wittenberg}
   \centerline{06099 Halle, Germany}
}

\bigskip

 \centerline{\emph{Dedicated to Jerry Goldstein on the occasion of his 70th birthday.}}

\begin{abstract}
We consider a quasilinear PDE system which models nonlinear vibrations
of a thermoelastic plate defined on a bounded domain in $\mR^n$. Well-posedness  of  solutions reconstructing  maximal parabolic regularity   in  nonlinear
thermoelastic plate  is established.
 In addition, exponential decay rates for strong solutions are also shown.
\end{abstract}

\begin{section}{Introduction}

  In this paper we study the existence and exponential stability of
  solutions to a quasilinear system arising in the modeling of
  nonlinear thermoelastic plates. The mathematical analysis of
  thermoelastic systems has attracted a lot of attention over the
  years.  An array of new and fundamental results in the area of
  wellposedness and stability of solutions to both linear and
  nonlinear thermoelasticity have been contributed to the field (see
  \cite{dafermos, dafermos1, racke, DenRac06, DenShiRac08, RivRac95, RivRac96} and
  references therein).

The focus of this paper is on {\it thermoelastic plates} and
associated uniform stability issues.  This particular class of
problems has received considerable attention in recent years,
particularly in the context of some new developments in control
theory. Questions such as exponential stability, controllability,
observability, unique continuation have been asked and partially
answered for both linear and nonlinear plates (see \cite{lagnese1} and
references therein). While there is at present  a vast literature dealing with well-posedness and
stability of linear  and semilinear thermoelastic equations (see above),
the treatment of quasilinear and fully nonlinear models defined on
multidimensional domains is much more  subtle and requires different
mathematical approaches.
 This paper deals with global and smooth
solutions defined for small initial data.

The equations we consider arise from a model that takes into account
the coupling between elastic, magnetic and thermal fields in a
{\em nonlinear} elastic plate model (see \cite{AmbBelMin83},
\cite{Bag99}, \cite{Ily48}, \cite{Lib75}, \cite{HasHovSasSte04}). In non-dimensional form,
the equations we consider are given below in \eqref{eq_0.1a}-\eqref{eq_0.1c}.
The nonlinearity
arises from the nature of the magnetoelastic material, owing to a
nonlinear dependence between the intensities of the deformation and
stress. We also assume that the material nonlinearity is cubic, as in the
original plate model \cite{HasHovSasSte04}. However, the arguments provided depend neither on structure of nonlinearity nor on the order near the origin. We put this generalization in evidence by considering a more general system under the sole Assumption 1 (see below).

Let $\Omega$ be a bounded domain of $\mathbb R^n$, $n\in\mathbb{N}$, with boundary $\partial \Omega\in C^2$. Consider the system
\begin{eqnarray}
\label{eq_0.1a}
& &
\left\{
\begin{array}{rr}
W_{tt}+\Delta^{2}W-\Delta \Theta+a\Delta ((\Delta W)^{3})=0 \\
\Theta_{t} -\Delta \Theta+\Delta W_{t} =0
\end{array} \right\} \mbox{ in } \Omega \times (0, T)
\\
\nonumber
&&
\\
\label{eq_0.1b}
&&
W=\Delta W=\Theta=0 \mbox{ on } \partial \Omega \times (0, T)
 \quad \textrm{(Boundary Conditions)}
\\
\nonumber
&&
\\
\label{eq_0.1c}
&&
\left\{ \begin{array}{lll}
W(x,0)=f(x)
& (x\in \Omega);

\\
W_{t}(x,0)=g(x)
& (x\in \Omega);

\\
\Theta(x,0)=h(x)
& (x\in \Omega);

\end{array} \right\} \quad \textrm{(Initial Conditions)}.
\end{eqnarray}
We assume that the material constant $a$ is positive.

In fact,  in what follows we will be able to obtain results for a more general version of equation (\ref{eq_0.1a})
where the cubic nonlinearity is replaced by a more general nonlinear function of superlinear growth.
More specifically, we consider
\begin{eqnarray}
\label{eq_0.1a-new}
& &
\left\{
\begin{array}{rr}
W_{tt}+\Delta^{2}W-\Delta \Theta+a\Delta (\phi (\Delta W))=0 \\
\Theta_{t} -\Delta \Theta+\Delta W_{t} =0
\end{array} \right\} \mbox{ in } \Omega \times (0, T)
\\
\nonumber
&&
\end{eqnarray}
 where  the  function $\phi$  satisfies:

\noindent\textbf{Assumption 1}:
$\phi \in C^{3-}(\mathbb{R}) $, $\phi(0) = \phi'(0) = \phi'' (0) =0$.

\end{section}

\begin{section}{Main Results}
\subsection{Notation}
Let $J = (0,T) $, where $T$ may be finite or $\infty$. For $p\in (1, \infty)$
we introduce the following function spaces
\begin{itemize}
\item $W_{p,0}^1(\Omega):=\overline{C_0^\infty(\Omega)}^{W_p^1}$.
\item $X_0 \equiv [L_p(\Omega)]^3 $,  $ X_1 \equiv [W^{2}_{p}(\Omega) \cap W_{p,0}^{1}(\Omega) ]^3  $.
\item
$ L_p(J, X_0) =[L_p(J,  L_p(\Omega) )]^3$ , $ L_p(J, X_1 ) = [L_p(J, W^{2}_{p}(\Omega) \cap W_{p,0}^{1}(\Omega) )]^3 $,
\item
$W^{1}_{p} (J, X_0 ) = [W^{1}_{p}(J, L_p(\Omega) )]^3 $
\item
$ X_{p} = ( X_0, X_1)_{1- \frac{1}{p} , p }=\begin{cases}[W_p^{2(1-1/p)}(\Omega)]^3,\ &\text{if}\ 1<p<3/2,\\
\{u\in [W_p^{2(1-1/p)}(\Omega)]^3:u|_{\partial\Omega}=0\},\ &\text{if}\ p>3/2.
\end{cases}$.
\end{itemize}
For $\mu\in (1/p,1]$ we set
\begin{itemize}
\item
$L_{p,\mu}(J;X_0):= \{u:J\to X_0:[t\mapsto t^{1-\mu}u(t)]\in L_p(J;X_0)\},$.
\item
$\mathbb{E}_{0,\mu}(J):=L_{p,\mu}(J;X_0)$
\item
$\mathbb{E}_{1,\mu}(J):= W_{p,\mu}^1(J;X_0)\cap L_{p,\mu}(J;X_1)$
\item $X_{p,\mu}$ is given by
\begin{align*} X_{p,\mu}:&=(X_0,X_1)_{\mu-1/p,p}\\
&=\begin{cases}[W_p^{2(\mu-1/p)}(\Omega)]^3,\ &\text{if}\ 1<\mu p<3/2,\\
\{u\in [W_p^{2(\mu-1/p)}(\Omega)]^3:u|_{\partial\Omega}=0\},\ &\text{if}\ \mu p>3/2.
\end{cases}
\end{align*}
\end{itemize}
Given $\omega\ge 0$, then
\begin{itemize}
\item $e^{-\omega}\mathbb{E}_{j,\mu}(J):=\{u\in \mathbb{E}_{j,\mu}(J):[t\mapsto e^{\omega t} u(t)]\in \mathbb{E}_{j,\mu}(J)\}$, $j\in\{0,1\}.$
\end{itemize}
If $X$ is some Banach space, then
\begin{itemize}
\item
$ e^{-\omega}BUC(J,X)=\{u:J\to X: [t\mapsto e^{\omega t}|u(t)|_X]\ \text{is bdd.\ and unif.\ cont.}\}$.
\item $e^{-\omega}C_0(\mathbb{R}_+;X):=\{u\in e^{\omega t}BUC(\mathbb{R}_+;X):e^{\omega t}|u(t)|_{X}\to 0\ \text{as}\ t\to\infty\}$.
\end{itemize}

\subsection{Formulation of the result}
\begin{theorem}\label{th:1}
 Let $n\in\mathbb{N}$, $p > 1 + \frac{n}{2} $ and $\mu\in (\frac{n+2}{2p},1]$. Assume that $\phi$ satisfies Assumption 1. With reference to the problem (\ref{eq_0.1b})-(\ref{eq_0.1a-new})
let $$x(0):=(\Delta  W(0), W_t(0), \Theta(0) ) \in
X_{p,\mu}.$$
Then the following assertions hold.
\begin{enumerate}
\item There exists $\rho > 0 $ such that
for all $|x(0)|_{X_{p,\mu}} \leq \rho $ and
 for every   $T > 0 $ there is a unique solution
 $x(t) = ( \Delta W(t), W_t(t), \Theta(t) )
 $ of (\ref{eq_0.1b})-(\ref{eq_0.1a-new}) with maximal parabolic regularity
$$ ( \Delta W , W_t, \Theta)\in [ L_{p,\mu}(J; W^{2}_p(\Omega))]^3 \cap [ W^{1}_{p,\mu}(J, L_p(\Omega) ) ]^3 \cap [BUC (J, W^{2\mu- 2/p}_p(\Omega)
) ]^3.$$
\item If in addition $p>(n+4)/2$, $\mu\in (\frac{n+4}{4p}+\frac{1}{2},1]$ and $\phi'(s)\ge 0$ for all $s\in\mathbb{R}$, then the same conclusion holds with no restriction on the size of initial data, provided $T>0$ is sufficiently small.
\item There exists $\omega > 0 $ and a constant $C> 0$ such that for $|x(0)|_{X_{p,\mu}} \leq \rho $ the following exponential estimate holds:
 $$| x(t) |_{X_{p,\mu}} \leq C e^{-\omega t } |x(0)|_{X_{p,\mu}} ,\quad t\ge 0.$$
\item For all $\sigma>0$ there exists $\omega > 0 $ (independent of $\sigma$)  and a constant $C(\sigma) > 0$ such that for $|x(0)|_{X_{p,\mu}} \leq \rho $ the following exponential decay rate holds:
 $$| x(t) |_{X_{p}} \leq C(\sigma) e^{-\omega t } |x(0)|_{X_{p,\mu}} ,\quad t\ge \sigma,$$
 where $C(\sigma)\to\infty$ as $\sigma\to 0$.
\end{enumerate}
\end{theorem}
By specializing $\phi$ to $\phi(s)=s^3 $ we obtain at once
\begin{corollary}\label{c:1}
The result stated in Theorem \ref{th:1} applies to the original model (\ref{eq_0.1a}).
\end{corollary}
\begin{remark}
The result obtained in Theorem \ref{th:1}  uses weighted norms $X_{p, \mu}$.
For $\mu =1 $ one obtains 'classical' $L_p$ estimates.
These norms account for singularity at the origin and provide trade-off between singularity and additional fractional regularity. Taking $p \rightarrow \infty$
allows to obtain "almost" $L_{\infty}$-estimates. This is reminiscent to some of the framework introduced in \cite{lunardi,lunardi-jde}
\end{remark}

\subsection{Comments}

\begin{enumerate}
\item
It is interesting to contrast the result of Theorem \ref{th:1} with the one of  Theorem 1.3 and Theorem 1.5 of \cite{nodea} obtained for the original model (\ref{eq_0.1a})-(\ref{eq_0.1c}) within the framework of
$L_2$  theory.  More specifically, in \cite{nodea}  global existence   and exponential decay rates are shown  in
 the so called finite energy which is $[L_2(\Omega)]^3 $ for the variable $x(t) $.  There is no {\it uniqueness}  result obtained within this framework.
This, of course, raises a  familiar dilemma of discrepancy between uniqueness and globality of solutions.
It is an interesting problem that is still open to the best knowledge of the authors.
\item
Unique and "small" solutions for equations (\ref{eq_0.1a})-(\ref{eq_0.1c})  have been also obtained in \cite{nodea} within the framework of maximal regularity with
the spaces $C^1(\bar{\Omega})$. However, the above framework leads to the "loss"of incremental differentiability
with respect to the initial data. This drawback is no longer present in Theorem \ref{th:1} where the space
$ X_{p,\mu}$ is invariant under the flow.
\item
One can consider more general structure of linear matrix operator in (\ref{eq_0.1a}) as long as
it is associated with an  exponentially stable semigroup. This is to say that the  coefficients
 of matrix $M$  introduced in  (\ref{A})  may be arbitrary  as long as all eigenvalues of $M$  have positive real parts.
\end{enumerate}
We shall next address the issue of higher regularity of solutions given by Theorem \ref{th:1}.
Among other things it will be shown below that under the additional assumption that $\phi \in C^{\infty} $, the solution $x(t)$ is infinitely many times differentiable in time away from $t=0$.
 \begin{theorem}\label{thm:2}
 Under the Assumptions of Theorem  \ref{th:1}  and with $\phi \in C^{\infty} (\mathbb{R}) $
  we obtain for all $k\in\mathbb{N}$ that $x^{(k)}\in e^{-\omega}C_0^{\infty}(J_\sigma, X_{p} ) $,
  for each $\sigma>0$, where $J_\sigma=[\sigma,\infty)$.
  \\
  In addition, if $[s\mapsto\phi(s)]$ is real analytic, then $[(0,\infty)\ni t\mapsto x(t)]$ is real analytic with values in $X_{p} $.
\end{theorem}

\end{section}
\begin{section}{Proof of Theorem \ref{th:1}}
The proof employs  techniques developed in the context of abstract parabolic problems and related maximal regularity.
\subsection
 { Abstract parabolic problems and maximal regularity. }
 Let $X$ be a given Banach space and  $J = [0, T] $ or $J = [0, \infty ) $ and let
 $A: D(A) \subset X \rightarrow X $ be a closed operator
 that is also densely defined .
 Consider an  abstract  Cauchy problem
 \begin{equation}\label{ev}
 u_t  =Au(t)+ f(t) ,\ t \in J,\quad u(0) = u_0.
 \end{equation}

 \begin{definition}
 We say that  $A$  admits  {\it maximal $L_p-regularity $ on $J$ } with some $p \in (1, \infty) $
 iff for each $f \in L_p(J; X ) $ and $u_0 =0 $, problem (\ref{ev}) admits a unique solution $u\in \mathbb{E}(J):=W_p^1(J;X)\cap L_p(J;D_A)$, where $D_A:=(D(A),|\cdot|_A)$.
  \end{definition}
  The space $\mathbb{E}(J)$ is continuously embedded into $BUC(J; \tr\mathbb{E} ) $ where the trace space $\tr\mathbb{E}$ is defined as
  $$ \tr \mathbb{E} = D_A(1-1/p,p ) = (X, D_A)_{1-1/p,p } $$
  and $(\cdot, \cdot )_{\theta, p} $ denotes the real interpolation method.

  \begin{definition}
  We say that the abstract inhomogeneous Cauchy problem {\it admits  maximal $L_p$ regularity},
  if the solution map
  $$(f, u_0 ) \mapsto u $$ is a topological isomorphism
  $$L_p(J;X) \times \tr \mathbb{E} \rightarrow \mathbb{E}(J) \subset BUC(J; \tr \mathbb{E}) $$
  \end{definition}
  In particular, the following estimate holds for operators $A$ with maximal $L_p$ regularity:
  \begin{equation}\label{max}
  |u|_{\mathbb{E}(J)} \leq M(J) ( |f|_{L_p(J;X)} + |u_0|_{\tr\mathbb{E} } ).
  \end{equation}

\subsection{Setting up (\ref{eq_0.1b})-(\ref{eq_0.1a-new}) as an abstract parabolic  problem}

We define \cite{liu-renardy,lt}
$$ U:=W_t,\ Z := \Delta W \ \text{and set}\ x:= ( Z  ,U,\Theta).$$ The differential operator
 $\Delta $, equipped with  zero Dirichlet boundary conditions,
generates an analytic semigroup on $L_p(\Omega) $. With the above notation,
the  original system can be written in the following operator form:
\begin{equation}
\label{eq:1}
x_t = \Delta  \left[\begin{array}{ccc} 0 & 1 & 0 \\ -1 & 0 & 1 \\ 0 & -1 & 1
\end{array} \right] x
- a \Delta   \left[ \begin{array}{c} 0 \\  \phi( Z  ) \\ 0 \end{array}
\right].
\end{equation}
Denoting
\begin{equation}
\label{A}
 A :=  \Delta  \left[\begin{array}{ccc} 0 & 1 & 0 \\ -1 & 0 & 1 \\ 0 & -1 & 1
\end{array}\right] = \Delta M
\end{equation}
where $M$ is $3\times 3 $ nonsingular matrix with eigenvalues having positive real parts.
It is easily  seen  that $A$ is the generator of an  exponentially stable analytic semigroup $e^{At}$ on
$ X_0 :=  L_p(\Omega) \times L_p(\Omega) \times L_p(\Omega) $ and
(\ref{eq:1}) can be rewritten as
\begin{equation}
\label{eq:2}
x_t = Ax + A F(x)
\end{equation}
where
\[
F(x) :=    a \left[ \begin{array}{ccc}   \phi(Z) &0&   0
\end{array}\right]^{\top} .
\]
Equation (\ref{eq:2}) is a nonlinear  abstract parabolic system   defined on $X_0$.
The nonlinearity enters via the generator $A$, and so
solvability of the system must depend  on ``maximal regularity'' properties
\cite{daprato,lunardi,PrSim04}. Since maximal regularity  does not hold within the context
of the $L_{\infty} ( [0, T]; X_0 ) $-topology \cite{lunardi}, one  should consider
the problem  within  the framework of $L_p$-spaces.

%

\medskip

\subsection{Representation as  a quasilinear abstract parabolic system}
Rewriting
$$ \Delta \phi(u) = \phi'(u) \Delta u + \phi''(u) |\nabla u |^2 ,$$
we obtain from  (\ref{eq:2}) that
$$ x_t = A x -a  \left[\begin{array}{ccc}
0,   \phi'(Z) \Delta Z  + \phi''(Z) |\nabla Z |^2 &
 ,0 \end{array}\right]^{\top}. $$
Denoting
\begin{eqnarray}\label{el}
 A(x)  := A - a  \left[\begin{array}{ccc}
0 & 0 & 0 \\  \phi'(Z) \Delta & 0
& 0 \\ 0 & 0& 0  \end{array}\right] = A + B(Z)  ,
\end{eqnarray}
leads  to  the consideration of   a quasilinear system  of the form:
\begin{equation}\label{quasi}
 x_t = A(x) x + f(x),
\end{equation}
where
\[
f(x) \equiv -a  \left[\begin{array}{ccc} 0,&   \phi''(Z) |\nabla Z |^2, & 0
\end{array}\right]^{\top}.
\]
Equation (\ref{quasi}) is a quasilinear abstract parabolic system.
Since $ A = M \Delta $ where $M$ is a real valued $3\times 3 $ matrix with eigenvalues possessing positive real parts, the operator $A$ has maximal parabolic regularity when considered on the space
$L_p(J, X_0 ) $ (see e.g.\ \cite{DHP03}). The interval $J$ can be extended to the positive real axis due to exponential stability of $e^{At}$.
This of  course implies that $ A(0) = A  $ enjoys maximal parabolic regularity on $J = (0, \infty ) = \mathbb{R}_+ $. By \cite{PrSim04} the operator $A$ has the property of maximal parabolic regularity in the weighted $L_p$-spaces
$$L_{p,\mu}(J;X_0):=\{u:J\to X_0:[t\mapsto t^{1-\mu}u(t)]\in L_p(J;X_0)\},$$
where $\mu\in (1/p,1]$. In particular, in \cite{PrSim04} the authors have shown that the problem
$$v_t=Av+f,\ v(0)=v_0$$
has a unique solution
$$v\in W_{p,\mu}^1(J;X_0)\cap L_{p,\mu}(J;X_1)=:\mathbb{E}_{1,\mu}(J)$$
if and only if $f\in L_{p,\mu}(J;X_0)=:\mathbb{E}_{0,\mu}(J)$ and
\begin{align*}
v_0\in X_{p,\mu}:&=(X_0,D(A))_{\mu-1/p,p}\\
&=\begin{cases}[W_p^{2(\mu-1/p)}(\Omega)]^3,\ &\text{if}\ 1<\mu p<3/2,\\
\{u\in [W_p^{2(\mu-1/p)}(\Omega)]^3:u|_{\partial\Omega}=0\},\ &\text{if}\ \mu p>3/2.
\end{cases}
\end{align*}
Moreover the estimate
$$|v|_{\mathbb{E}_{1,\mu}(J)}\le C(|f|_{\mathbb{E}_{0,\mu}(J)}+|v_0|_{X_{p,\mu}})$$
holds for some constant $C>0$.

Let $s(A)<0$ be the spectral bound of $A$
and let $f\in e^{-\omega}L_{p,\mu}(J;X_0)$ as well as $v_0\in X_{p,\mu}$ be given. Consider the problem
\begin{equation}\label{step2:eq1}
v_t=Av+f,\ v(0)=v_0
\end{equation}
in $e^{-\omega}L_{p,\mu}(J;X_0)$. The scaled function $u(t)=e^{\omega t}v(t)$ then solves the problem
\begin{equation}\label{step2:eq2}
u_t=(A+\omega)u+e^{\omega t}f,\ u(0)=v_0.
\end{equation}
Note that $s(A+\omega)=s(A)+\omega<0$ if $\omega\in [0,-s(A))$. Since by assumption $e^{\omega t}f\in L_{p,\mu}(J;X_0)$ and $v_0\in X_{p,\mu}$ it follows that there exists a unique solution $u\in \mathbb{E}_{1,\mu}(J)$ of \eqref{step2:eq2}. But this in turn implies that there exists a unique solution $v\in e^{-\omega}\mathbb{E}_{1,\mu}(J)$ of problem \eqref{step2:eq1} satisfying the estimate
 $$|v|_{e^{-\omega}\mathbb{E}_{1,\mu}(J)}\le C(|f|_{e^{-\omega}\mathbb{E}_{0,\mu}(J)}+|v_0|_{X_{p,\mu}}).$$
In other words we have shown that the operator $A$ has maximal parabolic regularity in the weighted spaces $e^{-\omega}L_{p,\mu}(J;X_0)$ as long as $\omega\in [0,-s(A))$ and $\mu\in (1/p,1]$.

The above allows to consider system (\ref{quasi}) within this maximal regularity framework.
In order to  be able to use maximal regularity theory we need to verify several assumptions regarding the operator $A(x) $ and the forcing term $f(x) $. This is done below.
\subsection{Supporting estimates}
We shall present several estimates which will be used later for the proof of main theorems.

\begin{lemma}\label{lem:estimates}
Let $ p > \frac{n+2}{2} $, $\mu\in (\frac{n+2}{2p},1]$ and $\omega\ge 0$ . Then
\begin{enumerate}
\item
The map $ (V,x) \mapsto \phi'(V)  \Delta x $ takes $$ e^{-\omega}BUC(J, W^{2(\mu-1/p)}_p(\Omega) )\times e^{-\omega}L_{p,\mu}(J, X_1 ) \rightarrow e^{-\omega}L_{p,\mu}(J;X_0))).$$
\item
The map $ ( V,Z)  \mapsto \phi'' (V) (\nabla V  \cdot \nabla Z) $ takes  $$e^{-\omega}BUC(J, W^{2(\mu-1/p)}_p(\Omega)
\times e^{-\omega}L_{p,\mu}(J, W^{2}_p(\Omega) ) \rightarrow
   e^{-\omega}L_{p,\mu}(J, L_p(\Omega) ) $$
\end{enumerate}
\end{lemma}
\begin{proof}
\begin{enumerate}
\item For $ p > \frac{n+2}{2} $ and $\mu\in (\frac{n+2}{2p},1]$ one has
$2( \mu- 1/p) - \frac{n}{p} > 0 $, hence
$$ W_p^{2(\mu-1/p)}(\Omega) \hookrightarrow L_{\infty}(\Omega) $$
Therefore  $\phi'(V) $ is a multiplier on $e^{-\omega}L_{p,\mu}(J, X_0) $. This along with the boundedness of
$\Delta : X_1 \rightarrow X_0 $ proves the claim.

\item Since we already know that $\phi''(V) $ is in $ e^{-\omega}BUC(J, L_{\infty}(\Omega) ) $ it suffices to analyze the mapping $ (V,Z)\mapsto  \nabla V
\cdot \nabla Z $. Our aim is to show that
\begin{equation}\label{eq-ZV}
  \nabla V \cdot \nabla Z \in e^{-\omega}L_{p,\mu} (J;L_p( \Omega))
  \end{equation}
   for
$V \in e^{-\omega}BUC(J, W^{2(\mu-1/p)}_p(\Omega) )$ and $Z \in e^{-\omega}L_{p,\mu}(J, W^{2}_p(\Omega)) $.\\
or alternatively
\begin{equation}\label{eq-ZV1}
\nabla V \in e^{-\omega}BUC(J, W^{2(\mu-1/p)-1}_p(\Omega) ),~ and ~\nabla Z \in e^{-\omega}L_{p,\mu}(J, W^{1}_p(\Omega))
\end{equation}

Applying H\"{o}lder's inequality with $r, \bar{r}$  exponents  yields
\begin{align}\label{H}
\begin{split}
\int_0^T \int_{\Omega} &e^{\omega pt}|\nabla Z|^p |\nabla V |^p t^{(1-\mu)p}dxdt\le\\ &\leq
|\nabla V|^{p}_{ L_{\infty} (J, L_{p \bar{r}}(\Omega))}   \int_0^T e^{\omega pt}|\nabla Z|^p_{L_{pr}(\Omega)}t^{(1-\mu)p}dt.
\end{split}
\end{align}
The choice of H\"{o}lder's exponent  will depend on the relation between $p$ and $n$.
Since $\mu>\frac{n+2}{2p}$, Sobolev's embeddings imply
\begin{equation}\label{Sob}
W^{2\mu-1-2/p}_p(\Omega)\hookrightarrow L_n(\Omega)
\end{equation}
Moreover
\begin{equation}\label{Sob1}
 W_p^1(\Omega)\hookrightarrow  \left\{\begin{array}{ccc}
 L_{np/(n-p)}(\Omega) &p < n \\ L_{\infty}(\Omega) & p > n\\
 L_q(\Omega) , q\in [1,\infty) & p =n  \end{array} \right.
\end{equation}
If $p<n$,we set $r=n/(n-p)$ and $\bar{r}=n/p$. Conversely, if $p>n$ then we choose $r=\infty$ and $\bar{r}=1$

In case $p=n$ we use (\ref{Sob}), (\ref{Sob1}) and the strict inequality for $\mu$. This yields $W_n^{2\mu-1-2/n}(\Omega)\hookrightarrow L_{n+\varepsilon}(\Omega)$ for a sufficiently small $\varepsilon>0$. Defining $\bar{r}:=(n+\varepsilon)/n>1$ and $q=pr=p\bar{r}/(\bar{r}-1)$ we finally obtain the desired estimate
\begin{align*}
\int_0^T \int_{\Omega} &e^{\omega pt}|\nabla Z|^{p}  |\nabla V |^pt^{(1-\mu)p} dxdt \\
&\leq
| V|^{p}_{ L_{\infty} (J, W_p^{2(\mu-1/p)}(\Omega)) }   \int_0^T e^{\omega pt}| Z|^p_{ W_p^2(\Omega)}t^{(1-\mu)p}dt,
\end{align*}
valid for all $p>1+n/2$ and $\mu\in (\frac{n+2}{2p},1]$.
\end{enumerate}
\end{proof}

\subsection{Solvability of a linear non-autonomous auxiliary problem}
  We begin with an auxiliary lemma which provides  solvability for the  linear equation with variable time and space coefficients.   The coefficients are assumed to be sufficiently smooth (in line with maximal parabolic regularity) and also of sufficiently small variation. The corresponding result is given below.

\begin{lemma}\label{lem:aux}
Let $ p > \frac{n+2}{2} $, $\mu\in (\frac{n+2}{2p},1]$, $\omega\in [0,-s(A))$ and\\ $V\in e^{-\omega} BUC(J;W_p^{2(\mu-1/p)}(\Omega))$ such that $\|V\|_{e^{-\omega}BUC(J;W_p^{2(\mu-1/p)})}\le\rho$. Then there exists $\rho_0>0$ such that for all $\rho\in (0,\rho_0)$ the linear problem
\begin{equation}\label{eq:2a}
 x_t = A(V) x + f(V,x) ,\ x(0) =x_0 \in X_{p,\mu}
 \end{equation}
with
$$ A(V) = A + B(V) , ~~B(V) = - a  \left[\begin{array}{ccc}
0 & 0 & 0 \\  \phi'(V) \Delta & 0
& 0 \\ 0 & 0& 0  \end{array}\right],$$
$$f(V,x) = -a [ 0, \phi'' (V) \nabla V \cdot \nabla Z , 0 ]^T $$
has a unique solution $x=(Z,U,\theta)\in e^{-\omega}\mathbb{E}_{1,\mu}(J)$ which satisfies the estimate
$$|x|_{e^{-\omega}\mathbb{E}_{1,\mu}}\le [C(\rho_0) +c]|x_0|_{X_{p,\mu}}.$$
where $C(\rho_0) \rightarrow 0 $ when $\rho_0 \rightarrow 0 $.
\end{lemma}
\begin{proof}
We first solve the problem
$$v_t=Av,\ v(0)=x_0$$
in $e^{-\omega}\mathbb{E}_{0,\mu}(J)$. This yields a solution $v = e^{At} x_0 \in e^{-\omega}\mathbb{E}_{1,\mu}(J)$ satisfying the estimate
\begin{equation}\label{v}
|v|_{e^{-\omega}\mathbb{E}_{1,\mu}}\le C|x_0|_{X_{p,\mu}}
\end{equation}
 Our next step is to homogenize the equation with respect to the initial data. For this we introduce change of variable
$w: = x-v $ , so that $w|_{t=0}=0$. Then the
sought after solution $x$  can be expressed as $ x: = w+v $ where $w$  solves
\begin{equation}\label{eq:2b}
 w_t = A(V) w + f(V,w)+g ,\ w(0) =0,
 \end{equation}
with  $g:=-B(V)v-f(V,v)\in e^{-\omega}\mathbb{E}_{0,\mu}(J)$ being a given function.
The regularity $g \in e^{-\omega}\mathbb{E}_{0,\mu}(J)$ follows directly from Lemma \ref{lem:estimates}.
Thus, our goal is reduced to establishing  well-posedness of (\ref{eq:2b}) .
 Writing  $A(V)=A+B(V)$, where $A$ has maximal parabolic regularity in $e^{-\omega}\mathbb{E}_{0,\mu}(J)$, we  may rewrite  linear equation in $w$  given in \eqref{eq:2b} as
$$w=(\partial_t-A)^{-1}[B(V)w+f(V,w)+g]$$
in the space $e^{-\omega}\,_0\mathbb{E}_{1,\mu}(J)$. By Lemma \ref{lem:estimates} and
 maximal parabolic regularity  of $A$, which then implies invertibility of  $(\partial_t-A)^{-1}$ from $e^{-\omega}\mathbb{E}_{0,\mu} $ into $ e^{-\omega}\mathbb{E}_{1,\mu} $, we obtain the estimate
\begin{align*}
|(\partial_t-A)^{-1}[B(V)w&+f(V,w)]|_{e^{-\omega}\mathbb{E}_{1,\mu}}\\
&\le C \Big(|\phi'(V)|_{L_\infty(L_\infty)}+\\
&\hspace{1cm}+|\phi''(V)|_{L_\infty(L_\infty)}
|V|_{e^{-\omega}BUC(W_p^{2(\mu-1/p)}(\Omega))}\Big)|w|_{e^{-\omega}\mathbb{E}_{1,\mu}}\\
&\le C|V|_{e^{-\omega}BUC(W_p^{2(\mu-1/p)}(\Omega))}^s|w|_{e^{-\omega}\mathbb{E}_{1,\mu}}\\
&\le C\rho^s|w|_{e^{-\omega}\mathbb{E}_{1,\mu}},
\end{align*}
 where $s > 0 $ by the assumption imposed on $\phi$. Therefore, if $\rho\in (0,\rho_0)$ and $\rho_0>0$ is sufficiently small, a Neumann series argument yields the statement. Recall  that equation for $w$ is linear ( $w\mapsto f(V,w)$ is  also linear).
 The above and maximal regularity  imply the estimate for $w$
 $$|w|_{e^{-\omega}\mathbb{E}_{1,\mu}} \leq C |g|{e^{-\omega}\mathbb{E}_{0,\mu}} \leq C \rho^s |v|_{e^{-\omega}\mathbb{E}_{1,\mu}}  \leq C \rho^s |x_0|_{X_{p, \mu}} $$
 The above estimate along with the estimate in (\ref{v}) leads to the final conclusion in the Lemma.
\end{proof}

\subsection{Analysis of nonlinear equation and completion of the proof}
 We shall apply Banach's fixed point theorem.
Let $\rho_0>0$ from the preceding lemma and
$$\cW:=\{W\in e^{-\omega}BUC(J;X_{p,\mu}):W(0)=x_0\ \text{and}\ |W|_{e^{-\omega}BUC(J;X_{p,\mu})}\le\rho\},$$
$\rho\in (0,\rho_0)$. For $W=(W_1,W_2,W_3)\in\cW$ define $\cT (W)=x$ to be the unique solution of \eqref{eq:2a}, where $V=W_1\in e^{-\omega}BUC(J;\tilde{X}_{p,\mu})$ and $\tilde{X}_{p,\mu}:=(L_p(\Omega),D(\Delta_D))_{\mu-1/p,p}$.

Note that $\cT$ is well-defined by Lemma \ref{lem:aux} and we have the estimate
\begin{equation}\label{eq:3}
|\cT(W)|_{e^{-\omega}\mathbb{E}_{1,\mu}}=|x|_{e^{-\omega}\mathbb{E}_{1,\mu}}\le [C(\rho_0) + c]|x_0|_{X_{p,\mu}}.
\end{equation}
From now on we assume that $|x_0|_{X_{p,\mu}}\le\delta$. This yields
$$|\cT(W)|_{e^{-\omega}BUC(J;X_{p,\mu})}\le M|\cT(W)|_{e^{-\omega}\mathbb{E}_{1,\mu}}\le M[C(\rho_0)+ c]\delta.$$
Here $M\ge 1$ denotes the embedding constant from $e^{-\omega}\mathbb{E}_{1,\mu}\hookrightarrow e^{-\omega}BUC(J;X_{p,\mu})$. Therefore, if $\delta=\rho/(M (C(\rho_0) + c) )$, it follows that $\cT(\cW)\subset \cW$.

 We shall now show that $\cT$ is a contraction on $\cW$. Let $W,\bar{W}\in\cW$ and $x=\cT(W)$, $\bar{x}=\cT(\bar{W})$. By the proof of Lemma \ref{lem:aux} we have
$$x-\bar{x}=(\partial_t-A)^{-1}[B(V)x-B(\bar{V})\bar{x}+f(V,x)-f(\bar{V},\bar{x})],$$
since $(x-\bar{x})(0)=0$. It follows that
$$|x-\bar{x}|_{e^{-\omega}\mathbb{E}_{1,\mu}}\le C(|B(V)x-B(\bar{V})\bar{x}|_{e^{-\omega}\mathbb{E}_{0,\mu}}
+|f(V,x)-f(\bar{V},\bar{x})|_{e^{-\omega}\mathbb{E}_{0,\mu}}).$$
For the first term on the right side we estimate as follows.
\begin{align*}
|B(V)x&-B(\bar{V})\bar{x}|_{e^{-\omega}\mathbb{E}_{0,\mu}}\le |B(V)(x-\bar{x})|_{e^{-\omega}\mathbb{E}_{0,\mu}}
+|(B(V)-B(\bar{V}))\bar{x}|_{e^{-\omega}\mathbb{E}_{0,\mu}}\\
&\le C[\rho^s|x-\bar{x}|_{e^{-\omega}\mathbb{E}_{1,\mu}}
+\rho|V-\bar{V}|_{e^{-\omega}BUC(J;\tilde{X}_{p,\mu})}
|\Delta\bar{x}|_{e^{-\omega}\mathbb{E}_{0,\mu}}]\\
&\le C[\rho^s|x-\bar{x}|_{e^{-\omega}\mathbb{E}_{1,\mu}}
+\rho|V-\bar{V}|_{e^{-\omega}BUC(J;\tilde{X}_{p,\mu})}
|\bar{x}|_{e^{-\omega}\mathbb{E}_{1,\mu}}]\\
&\le C[\rho^s|x-\bar{x}|_{e^{-\omega}\mathbb{E}_{1,\mu}}
+\rho^2 |V-\bar{V}|_{e^{-\omega}BUC(J;\tilde{X}_{p,\mu})}]\\
&\le C[\rho^s|x-\bar{x}|_{e^{-\omega}\mathbb{E}_{1,\mu}}+ \rho^2
|W-\bar{W}|_{e^{-\omega}BUC(J;X_{p,\mu})}]
\end{align*}
for some $s>0$, since $|V|_{e^{-\omega}BUC(J;\tilde{X}_{p,\mu})},|\bar{V}|_{e^{-\omega}BUC(J;\tilde{X}_{p,\mu})},
|\bar{x}|_{e^{-\omega}\mathbb{E}_{1,\mu}}\le \rho$ and $\tilde{X}_{p,\mu}\hookrightarrow L_\infty(\Omega)$. In a similar way we obtain
$$|f(V,x)-f(\bar{V},\bar{x})|_{e^{-\omega}\mathbb{E}_{0,\mu}}\le C\rho^s\left(|x-\bar{x}|_{e^{-\omega}\mathbb{E}_{1,\mu}}
+|W-\bar{W}|_{e^{-\omega}BUC(J;X_{p,\mu})}\right),$$
for some $s>0$, by Assumption 1.
Since
$$|\cT(W)-\cT(\bar{W})|_{e^{-\omega}BUC(J;X_{p,\mu})}\le M|\cT(W)-\cT(\bar{W})|_{e^{-\omega}\mathbb{E}_{1,\mu}}=
M|x-\bar{x}|_{e^{-\omega}\mathbb{E}_{1,\mu}}$$
it follows that $\cT$ is a strict contraction on $\cW$ provided that $\rho>0$ is sufficiently small. The contraction mapping principle yields a unique fixed point $x_*\in\cW$ of $\cT$, i.e.\ $\cT(x_*)=x_*$. By construction of $\cT$, the fixed point $x_*$ is the unique solution of \eqref{quasi} in $e^{-\omega}\mathbb{E}_{1,\mu}$. Moreover $x_*$ satisfies
$$|x_*|_{e^{-\omega}C_0(J;X_{p,\mu})}\le M_1|x_*|_{e^{-\omega}\mathbb{E}_{1,\mu}(J)}\le M_1[C(\rho_0) +c] |x_0|_{X_{p,\mu}},$$
as well as
\begin{multline*}
|x_*|_{e^{-\omega}C_0(J_\sigma;X_{p})}\le M_2|x_*|_{e^{-\omega}\mathbb{E}_{1}(J_\sigma)}\le\\
\le M_2\sigma^{\mu-1}|x_*|_{e^{-\omega}\mathbb{E}_{1,\mu}(J)}\le M_2\sigma^{\mu-1}[C(\rho_0) +c] |x_0|_{X_{p,\mu}},
\end{multline*}
where $J_\sigma=[\sigma,\infty)$ for some $\sigma>0$ and $J=J_0$. Here the constant $M_1>0$ comes from the embedding (see \cite{PrSim04})
$$e^{-\omega}\mathbb{E}_{1,\mu}(J)\hookrightarrow e^{-\omega}C_0(J;X_{p,\mu})$$
and the constant $M_2>0$ does not depend on $\sigma>0$. This can be seen as follows.
\begin{align*}
|x_*|_{e^{-\omega}\mathbb{E}_{1}(J_\sigma)}^p&=\int_\sigma^\infty e^{\omega p t}|x_*(t)|_{X_1}^pdt+\int_\sigma^\infty e^{\omega p t}|\dot{x}_*(t)|_{X_0}^pdt\\
&=e^{\omega p\sigma}\left(\int_0^\infty e^{\omega p \tau}|x_*(\tau+\sigma)|_{X_1}^pd\tau+\int_0^\infty e^{\omega p \tau}|\dot{x}_*(\tau+\sigma)|_{X_0}^pd\tau\right)\\
&=e^{\omega p\sigma}|T(\sigma)x_*|_{e^{-\omega}\mathbb{E}_{1}(J)}^p\ge\frac{1}{M_2^p}[\sup_{\tau\ge 0}e^{\omega p(\tau+\sigma)}|x_*(\tau+\sigma)|_{X_p}]^p\\
&=\frac{1}{M_2^p}[\sup_{t\ge \sigma}e^{\omega pt}|x_*(t)|_{X_p}]^p
\end{align*}
where $[T(\sigma)f](\tau):=f(\tau+\sigma)$, $\tau,\sigma\ge 0$, is the semigroup of left-translations and $M_2>0$ denotes the embedding constant associated to
$$e^{-\omega}\mathbb{E}_{1}(J)\hookrightarrow e^{-\omega}C_0(J;X_p).$$
This yields the estimates
$$|x_*(t)|_{X_{p,\mu}}\le M_1[C(\rho_0) +c] e^{-\omega t}|x_0|_{X_{p,\mu}},\quad t\ge 0,$$
and
$$|x_*(t)|_{X_{p}}\le M_2\sigma^{\mu-1}[C(\rho_0) +c] e^{-\omega t}|x_0|_{X_{p,\mu}},\quad t\ge \sigma>0,$$
valid for all $|x_0|_{X_{p,\mu}}\le\delta$. It follows that $x_*(t)\to 0$ in $X_{p}$ as $t\to\infty$ at an exponential rate and the trivial equilibrium of \eqref{quasi} is exponentially stable in $X_{p,\mu}$ for $\mu \in (\frac{n+2}{2p},1] $. This proves assertion (1), (3) \& (4).

If in addition $p>(n+4)/2$, $\mu\in (\frac{n+4}{4p}+\frac{1}{2},1]$ and $\phi'(s)\ge 0$ for all $s\in\mathbb{R}$ and $x_0$ is not necessarily small in $X_{p,\mu}$, then one can show that there exists a possibly small $T=T(x_0)>0$ such that (\ref{eq_0.1b})-(\ref{eq_0.1a-new}) has a unique solution
$$ ( \Delta W , W_t, \Theta)\in [ L_{p,\mu}(J; W^{2}_p(\Omega))]^3 \cap [ W^{1}_{p,\mu}(J, L_p(\Omega) ) ]^3 \cap [BUC (J, W^{2\mu- 2/p}_p(\Omega)
) ]^3.$$
$J=[0,T]$. This follows from the lines of the proof of \cite[Theorem 2.1]{KPW10}, hence assertion (2) follows.
\end{section}

\section{Proof of Theorem \ref{thm:2}}
The proof of Theorem \ref{thm:2} follows from  Theorem \ref{th:1} and a  suitable application of the implicit function theorem (see \cite{Deim85}), which gives both differentiability and analyticity of the nonlinear flow. The parameter trick which will be applied below goes back to \cite{Ang90} and in the context of maximal regularity it has been applied e.g.\ in \cite{Pr02}.

\subsection{Differentiability of solutions}
We will show that
$$[x\mapsto (A(x),f(x))]\in C^1(e^{-\omega}\mathbb{E}_{1,\mu}(J);e^{-\omega}BUC(J;\mathcal{L}(X_1,X_0))\times e^{-\omega}\mathbb{E}_{0,\mu}(J)).$$
where $$A (x) = A  +  B(Z)\ \text{and}\ f(x) = -a [ 0, \phi''(Z) |\Delta Z|^2 , 0 ]^{\sf T} $$
To this end, let $\phi$ satisfy Assumption 1 and, in addition, assume that $\phi\in C^3(\mathbb{R})$. The natural candidate for $f'(x_*)x$ is
$$f'(x_*)x=-a
\begin{bmatrix}
0 \\ \phi'''(Z_*)Z|\nabla Z_*|^2+2\phi''(Z_*)(\nabla Z_*,\nabla Z) \\ 0
\end{bmatrix}.
$$
We have
$$f(x_*+x)-f(x_*)-f'(x_*)x=-a[0,f_1(Z,Z_*)+f_2(Z,Z_*)+f_3(Z,Z_*),0]^{\sf T},$$
where
$$f_1(Z,Z_*):=(\phi''(Z_*+Z)-\phi''(Z_*)-\phi'''(Z_*)Z)|\nabla Z_*|^2$$
$$f_2(Z,Z_*):=2(\nabla Z_*|\nabla Z)(\phi''(Z_*+Z)-\phi''(Z_*))$$
and $$f_3(Z,Z_*):=\phi''(Z_*+Z)|\nabla Z|^2$$ Since $\phi\in C^3(\mathbb{R})$, it is easy to check the desired $C^1$-property for $[x\mapsto f(x)]$ with the help of Lemma \ref{lem:estimates}. In the same way (which is actually easier) one can show that $[x\mapsto A(x)]$ with $A(x):=A+B(x)$ is $C^1$ with
$$[B'(x_*)x]x_*=-a
\begin{bmatrix}
0 \\ \phi''(Z_*)Z\Delta Z_*\\0
\end{bmatrix}$$
Let $x_*(t)$ be the solution according to Theorem \ref{th:1}. We introduce a new function $x_\lambda(t):=x_*(\lambda t)$ for $\lambda\in (1-\epsilon,1+\epsilon)$ and $t\in J$. It follows that $\partial_t x_\lambda=\lambda(\partial_t x_*)(\lambda t)$, hence
$$\partial_t x_\lambda(t)+\lambda A(x_\lambda(t))x_\lambda(t)=\lambda f(x_\lambda (t)),\ t\in J,\quad x_\lambda(0)=
x_(0).$$
Define a mapping $H:(1-\epsilon,1+\epsilon)\times e^{-\omega}\mathbb{E}_{1,\mu}(J)\to e^{-\omega}\mathbb{E}_{0,\mu}(J)\times X_{p,\mu}$ by
$$H(\lambda,x)=(\partial_t x+\lambda A(x)x-\lambda f(x),x(0)-x_*(0))$$
Note that $H(1,x_*)=0$, $H\in C^1((1-\epsilon,1+\epsilon)\times e^{-\omega}\mathbb{E}_{1,\mu}(J))$ and
$$D_x H(\lambda,x_*)x=(\partial_t x+\lambda [A'(x_*)x]x_*+\lambda A(x_*)x-\lambda f'(x_*)x,x(0)),$$
by the differentiability properties of $A$ and $f$. This yields
$$D_x H(1,x_*)x=(\partial_t x+[B'(x_*)x]x_*+A(x_*)x-f'(x_*)x,x(0)),$$
where as before $A(x)=A+B(x)$. We already know that $$|x_*|_{e^{-\omega}\mathbb{E}_{1,\mu}(J)}\le C|x_*(0)|_{X_{p,\mu}}$$
by \eqref{eq:3}. This yields
\begin{align*}
|[B'(x_*)x]x_*|_{e^{-\omega}\mathbb{E}_{0,\mu}(J)}&\le a|\phi''(Z_*)Z\Delta Z_*|_{e^{-\omega}\mathbb{E}_{0,\mu}(J)}\\
&\le a|\phi''(Z_*)|_{e^{-\omega} BUC(J\times\Omega)}|Z|_{e^{-\omega} BUC(J\times\Omega)}|Z_*|_{e^{-\omega}\mathbb{E}_{1,\mu}(J)}\\
&\le C|x_*|_{e^{-\omega}BUC(J;X_p)}^s|x|_{e^{-\omega}\mathbb{E}_{1,\mu}(J)}
|x_*|_{e^{-\omega}\mathbb{E}_{1,\mu}(J)}\\
&\le C\delta^{1+s}|x|_{e^{-\omega}\mathbb{E}_{1,\mu}(J)},
\end{align*}
if $|x_*(0)|_{X_{p,\mu}}\le\delta$. Similarly one can show that
$$|B(x_*)x-f'(x_*)x|_{e^{-\omega}\mathbb{E}_{0,\mu}(J)}\le C\delta^s|x|_{e^{-\omega}\mathbb{E}_{1,\mu}(J)}.$$
It follows that if $\delta>0$ is sufficiently small, then the linear operator
$$[x\mapsto [B'(x_*)x]x_*+B(x_*)x-f'(x_*)x]$$
is a small perturbation of $[x\mapsto Ax]$. A Neumann series argument as in Lemma \ref{lem:aux} yields that the operator $D_xH(1,x_*):e^{-\omega}\mathbb{E}_{1,\mu}(J)\to e^{-\omega}\mathbb{E}_{0,\mu}(J)\times X_{p,\mu}$ is an isomorphism. Therefore, by the implicit function theorem (see e.g.\ Theorem 15.1 in \cite{Deim85}), we obtain a $C^1$-mapping $\Phi:(1-\eta,1+\eta)\to e^{-\omega}\mathbb{E}_{1,\mu}(J)$ such that $H(\lambda,\Phi(\lambda))=0$ for all $\lambda\in (1-\eta,1+\eta)$ and $\Phi(1)=x_*$. By uniqueness it follows that
$x_{\lambda}(t) = x_{*}(\lambda t ) $ agrees with $\Phi(\lambda )(t) $, hence  $x_\lambda=\Phi(\lambda)\in e^{-\omega}\mathbb{E}_{1,\mu}(J)$ is $C^1$ in $\lambda\in (1-\eta,1+\eta)$ with derivative
$\partial_\lambda x_\lambda(t)=t(\partial_tx_*)(\lambda t)$. Evaluating at $\lambda=1$ yields $[t\mapsto t\partial_tx_*(t)]\in e^{-\omega}\mathbb{E}_{1,\mu}(J)$, hence $\partial_t (t x_*)=t\partial_t x_*+x_*\in e^{-\omega}\mathbb{E}_{1,\mu}(J)\hookrightarrow e^{-\omega}C_0(J;X_{p,\mu})$. If we restrict ourselves to intervals $J_\sigma:=[\sigma,\infty)$, $\sigma>0$, we may drop the factor $t$ to obtain
$$x_* \in e^{-\omega}[W_p^2(J_\sigma;X_0)\cap W_p^1(J_\sigma;X_1)\cap C_0^1(J_\sigma;X_{p})].$$
\subsection{Higher order differentiability}
If $\phi\in C^{k+2}(\mathbb{R})$, $k\in\mathbb{N}\cup\{\infty\}$ then one can show that $A$ and $f$ are $C^k$, hence $H$ is $C^k$. Therefore, a corollary of the implicit function theorem (see e.g.\ Corollary 15.1 in \cite{Deim85}) implies that $\Phi$ is $C^k$, hence $[\lambda \mapsto x_\lambda=\Phi(\lambda)]$ is $C^k$ and then, by computing reiterated derivatives, we obtain  $[t\mapsto t^k x_*^{(k)}(t)]\in e^{-\omega}\mathbb{E}_{1,\mu}(J)$ for all $k\in \mathbb{N}$. Consequently, we obtain
$$x_* \in e^{-\omega}[W_p^{k+1}(J_\sigma;X_0)\cap W_p^k(J_\sigma;X_1)\cap C_0^k(J_\sigma;X_{p})].$$
\subsection{Analyticity}
If $\phi$ is even real analytic, then $A$ and $f$ are real analytic and then so is $H$. The real analytic version of the implicit function theorem (see e.g.\ \cite[Theorem 15.3]{Deim85}) yields that $\Phi$ is real analytic, hence $[\lambda\mapsto x_\lambda]$ is real analytic. Let $t_0>0$ be fixed and define $e(x):=x(t_0)$. It is easy to see that $e\in\mathcal{L}(e^{-\omega}\mathbb{E}_{1}(J_\sigma);X_{p})$, hence
$[\lambda\mapsto x_\lambda(t_0)=x_*(\lambda t_0)]$ is real analytic.
But since this is true for any $t_0>0$, we obtain that $x_*$ is real analytic for all $t>0$ with values in $X_{p}$.



\medskip

\end{document}